\documentclass[12pt]{hha}
\usepackage{amsmath,amstext,amsfonts,amsthm,amssymb,epigraph}
\usepackage{eucal}
\usepackage{framed}
\usepackage[all]{xy}                                        %
\swapnumbers

\newtheorem{lem}[subsubsection]{Lemma}
\newtheorem{prp}[subsubsection]{Proposition}
\newtheorem{crl}[subsubsection]{Corollary}

\newtheorem{LEM}[subsection]{Lemma}
\newtheorem{PRP}[subsection]{Proposition}
\newtheorem{CRL}[subsection]{Corollary}

\newtheorem*{Lem}{Lemma}
\newtheorem*{Prp}{Proposition}
\newtheorem*{Crl}{Corollary}

{  \theoremstyle{definition}
           \newtheorem{dfn}[subsubsection]{Definition}

           \newtheorem*{Dfn}{Definition}

           \newtheorem*{Rems}{Remarks}
}

\newcommand{\Cat}{\mathtt{Cat}}

\newcommand{\Cor}{\mathtt{Cor}}
\newcommand{\coc}{\mathit{coc}}
\newcommand{\Coc}{\mathtt{Coc}}
\newcommand{\colim}{\operatorname{colim}}

\newcommand{\Del}{\mathit{Del}}

\newcommand{\eq}{\mathit{eq}}

\newcommand{\Fun}{\operatorname{Fun}}

\newcommand{\Hom}{\mathrm{Hom}} 
\newcommand{\Ho}{\operatorname{Ho}}

\newcommand{\id}{\mathrm{id}}

\newcommand{\Left}{\mathbf{L}}

\newcommand{\Map}{\operatorname{Map}}
\mathchardef\mhyphen="2D

\newcommand{\Ob}{\operatorname{Ob}}
\newcommand{\op}{\mathrm{op}}

\newcommand{\rlarrows}{\substack{\longrightarrow\\ \longleftarrow}}
\newcommand{\Right}{\mathbf{R}}
\newcommand{\RE}{\mathtt{RE}}

\newcommand{\Set}{\mathtt{Set}}

\newcommand{\cC}{\mathcal{C}}
\newcommand{\cD}{\mathcal{D}}
\newcommand{\cE}{\mathcal{E}}
\newcommand{\cF}{\mathcal{F}}

\newcommand{\cL}{\mathcal{L}}
\newcommand{\cM}{\mathcal{M}}
\newcommand{\cN}{\mathcal{N}}

\newcommand{\cS}{\mathcal{S}}

\newcommand{\cW}{\mathcal{W}}
\newcommand{\cX}{\mathcal{X}}

\begin{document}

\title{So, what is a derived functor?}
\author{Vladimir Hinich}
\email{hinich@math.haifa.ac.il}
\address{Department of Mathematics, University of Haifa,
Mount Carmel, Haifa 3498838,  Israel}
\received{Month Day, Year}
\revised{Month Day, Year}
\published{Month Day, Year}
\submitted{Emily Riehl}
\volumeyear{0} 
\volumenumber{0} 
\issuenumber{0}   
\startpage{0}     
\articlenumber{0} 
 
\begin{abstract}We rethink the notion of derived functor in terms of correspondences, that is, functors $\cE\to[1]$. While derived functors in our sense, when  exist, 
are given by Kan extensions,  their existence is a 
strictly stronger property than the existence of Kan extensions. 
We show, however, that derived functors exist in the cases one expects them to exist.
Our definition is especially convenient for the description 
of a passage from an adjoint pair $(F,G)$ of functors to a 
derived adjoint pair $(\Left F,\Right G)$. In particular, 
canonicity of such a passage is immediate in our approach.
Our approach makes perfect sense in the context of $\infty$-categories.
\end{abstract}
\maketitle
\section{Introduction}
\epigraph{ This is a new rap on the oldest of stories —-- \\ Functors on abelian categories. \\ If the functor is left exact \\ You can derive it and that's a fact.\\ But first you must have enough injective\\ Objects in the category to stay active.\\ If that's the case — no time to lose;\\  Resolve injectively any way you choose.\\  Apply the functor and don't be sore —--\\  The sequence ain't exact no more.\\  Here comes the part that is the most fun, Sir,\\  Take homology to get the answer.\\  On resolution it don't depend:\\  All are chain homotopy equivalent.\\  Hey, Mama, when your algebra shows a gap\\  Go over this Derived Functor Rap.}
{ P.~Bressler, Derived Functor Rap, 1988}
\addtocounter{subsection}{-1}
\subsection{Advertisement}
\label{ss:advertisement}
Our approach to derived functors can be explained in 
one sentence.

\begin{framed}
\begin{center}
{\sl In the language of cocartesian
fibrations over $[1]$, calculation of a left derived functor
becomes a localization.}
\end{center}
\end{framed}
This sentence is, actually, a recipe:
\begin{itemize}
\item Convert a functor $f:\cC\to\cD$ into a cocartesian fibration $p:\cE\to[1]$.
\item Localize $\cE$.
\item If the localization $\cE'$ of $\cE$ remains a cocartesian fibration over $[1]$, we say that $f$ has a left derived functor; this is the functor classifying $\cE'$.
\end{itemize}

One similarly treats the right derived functors as well as 
the derived functors of an adjoint pair of functors. 
 
In this paper we argue that this approach leads to a
very good notion of derived functor. We show that,
given an adjoint pair of functors, their respective derived 
functors, if exist, are automatically adjoint. We also show 
that the derived functors defined in this way behave nicely 
in families, as explained in \ref{sss:intro-families} below.

\subsection{A bit of history}

The prehistoric understanding of derived functors, based on 
existence of resolutions, is beautifully described in the 
epigraph. While this description makes perfect sense, it 
cannot possibly serve as a definition; it is merely a 
construction. 

A historic period starts with reformulation of derived 
functors in terms of localization of categories, performed by Grothendieck and Verdier in the abelian setting, and by Quillen in topology. It was first documented in Hartshorne's notes ~\cite{RD}. The idea of this approach is
that, in order to make sense of arbitrary choise of resolutions, one has to construct a category where an object
of an abelian category and its resolution become isomorphic; this is the derived category, and it is constructed by localizing the category of complexes. A similar idea
led Quillen~\cite{Q.HA} to define model categories and their localizations, homotopy categories.

Using the language of localization, derived functor are 
defined by a universal property: according to 
Hartshorne~\cite{RD}, 1.5, and Quillen~\cite{Q.HA}, 4.1, a 
left derived functor can be defined as (what is nowadays 
called) the right Kan extension, whereas a right derived 
functor can be defined as the left Kan extension.

A similar approach is used for defining derived functors in the context of $\infty$-categories,
see~D.-C. Cisinski~\cite{C}.

Another definition of derived functor was suggested by P.~Deligne in his report on \'etale cohomology with proper support, see~\cite{D}, in the context of triangulated categories. The value of a left derived 
functor, according to Deligne, is a pro-object of the 
respective localization. If the values of a derived functor
so defined are corepresentable, then the derived functor is
a right Kan extension of the original functor. However, the existence of right Kan extension does not seem to imply
corepresentability of Deligne's derived functor.

\subsection{}
The definition of derived functors via Kan extensions is 
not, in our opinion, fully satisfactory. Here is one of the 
problems. The functors one has to derive often come in 
pairs~\footnote{adjoint pairs.}. And, given an adjoint pair 
of functors, one expects them to give rise to an adjoint 
pair of derived functors. Each separate derived functor has 
a universal description as a Kan extension; but adjoint pair 
is not just a pair of functors: to define an adjunction one 
also needs to specify a unit or a counit of the adjunction. 
This cannot be deduced in general from the description of 
derived functors as Kan extensions.

\subsection{Summary}

In Subsection~\ref{ss:def-derived} we define, following the recipe explained in~\ref{ss:advertisement},
left and right derived functors. Our definition immediately
implies that, for $f$ left adjoint to $g$, if $\Left f$ and $\Right g$ exist, they are automatically adjoint.

In Section~\ref{sec:corr-ext} we describe the category of 
correspondences and its full subcategory of cocartesian correspondences. 

The main results are proven in Section~\ref{sec:main}.

They include:

Corollary 4.2 saying that the left derived functor,  in the 
sense of Definition~\ref{dfn:left}, if exists, is a right 
Kan extension. 

Sufficient condition of existence of left (right) derived 
functor, including the case of left (right) Quillen 
functors.

In \ref{ss:functoriality} we study the properties of diagrams of derived functors. The details are explained in~\ref{sss:intro-families} below.

In~\ref{ss:deligne} we show that Deligne's definition of
derived functor given in~\cite{D}, is a special case of our definition. 

In~\ref{ss:infty-conv} we show that, if $f':\cC'\to\cD'$
is $\infty$-categorical derived functor, passage to the respective homotopy categories defines a derived functor
with respect to conventional localizations.

\begin{Rems} 
\vbox{\quad   }
\begin{itemize}
\item[1.]
The derived functors defined by Deligne
are known to automatically preserve adjunction, see B.~Keller's~\cite{K}, 13.6.
\item[2.] In 2007 George Maltsiniotis~\cite{Mal} made a beautiful observation: if Kan extensions $\Left f$ and 
$\Right g$ of an adjoint pair of functors $(f,g)$ are {\sl absolute}, $\Left f$ and $\Right g$ acqure an automatic 
adjunction. The same holds for an adjunction of infinity categories, see~\cite{C}. In particular, Quillen's adjunction leads to
absolute Kan extensions, so the derived functors between the
infinity categories underlying a Quillen pair, are automatically adjoint. Derived functors in the sense of Definition~\ref{dfn:left} are, actually, absolute.

We are not sure, however, that defining
derived functors as absolute Kan extensions would yield a notion preserving diagrams of derived functors as in ~\ref{sss:intro-families} below.
\end{itemize}
\end{Rems}

\subsubsection{Diagrams of derived functors} 
\label{sss:intro-families}
It is known that derived functor of a composition is not necessarily a composition of derived functors; it turns out, however, that this is the only obstacle to functoriality of the passage to derived functors; in 
Section~\ref{ss:functoriality} we prove that,
given a ($\infty$-) functor $F:B\to\Cat$ and an appropriate collection of ($\infty$-) subcategories $\cW_b\subset F(b),\ b\in B$, so that 
for any arrow $a:b\to b'$ in $B$ the derived functor of
$F(a)$ exists, and is compatible with compositions, then
the family of derived functors $\Left F(a)$ ``glue'' to
a   functor $\Left F: B\to\Cat$ carrying $b\in B$ to
the localization $F(b)'$ of $F(b)$ with respect to $\cW_b$,
and any arrow $a:b\to b'$ to $\Left F(a):F(b)'\to F(b')'$.

We expect this property will be useful in studying  higher descent in style of~\cite{HS} and \cite{Me}.

\subsection{Acknowledgements} This work, being formally 
independent of George Maltsiniotis' \cite{Mal}, stems from a 
similar dissatisfaction with the existing definition of 
derived functor. We are grateful to him for bringing our attention to this work. We are grateful to 
B.~Keller who explained us that Delgne's definition  leads to  automatical adjunction of the derived functors. We are also grateful to D.-C.~Cisinski who informed us 
about his book~\cite{C}. Discussions with Ilya Zakharevich 
were very useful. We are very grateful to the referee 
for his request to clarify some sloppy passages in the
original version of the paper.
The present work was partially supported 
by ISF grants 446/15 and 786/19.

\section{Left and right derived functors}

We will now present our definition of derived functors and formulate the main results. 

\

In what follows the word ``category'' means infinity 
category, and ``conventional category'' means a category in 
the conventional sense. Our way of dealing with infinity
categories is model-independent in the following sense.
We work in the (infinity) category of (infinity) categories
$\Cat$. All properties of categories or of morphisms in 
categories we use are invariant under equivalences. 
{\sl In particular, all limits and colimits in this paper are in the sense of infinifty categories.} Since
Quillen equivalence of  model categories gives rise to
equivalent infinity categories, and all existing models
for infinity categories are Quillen equivalent to each other, we can use existing results proven in specific models
to claim properties of our ``model-independent'' $\Cat$.
See more details about this approach in~\cite{H.EY}, 
Section 2.

\subsection{}
A functor $f:\cC\to\cD$ can be converted, via Grothendieck 
construction, to a cocartesian fibration $p:\cE\to[1]$.

One has $\cE=(\cC\times[1])\sqcup^\cC\cD$, with the 
map $\cC\to\cD$ given by $f$, see~\cite{H.lec}, 9.8. Here is 
a description of $\cE$ for conventional categories.

\begin{itemize}
\item $\Ob(\cE)=\Ob(\cC)\sqcup\Ob(\cD)$.
\item $\Hom_\cE(x,y)=
\begin{cases}
\Hom_\cC(x,y), & x,y\in\cC,\\
\Hom_\cD(x,y), & x,y\in\cD,\\
\Hom_\cD(f(x),y), & x\in\cC, y\in\cD,\\
\emptyset, & x\in\cD, y\in\cC.
\end{cases}$
\end{itemize}
The cocartesian fibration classified by the functor
$f:\cC\to\cD$ will be denoted $p_f:\cE_f\to[1]$.

\subsection{}
\label{ss:f}
Similarly, a functor $f:\cC\to\cD$ can be converted to
a cartesian fibration $q_f:\cF_f\to[1]$ given by the formula
$\cF_f=\cD\sqcup^\cC(\cC\times[1])$. Its fibers at $0$ and
$1$ are $\cD$ and $\cC$ respectively.

\subsection{Derived functors}
\label{ss:def-derived}

\subsubsection{Localization}
\label{sss:loc}
Recall that a localization of a category $\cC$ with respect
to its subcategory $\cW$ (we assume $\cM$ contains
the maximal subspace $\cC^\eq$ of $\cC$) is defined by a universal property: a functor $\cC\to\cC[\cW^{-1}]$ is an
initial object among functors carrying $\cW$ to equivalences. Localization is functorial in the following sense. Define $\Cat^+$ the category of marked categories,
that is pairs $(\cC,\cC^\circ)$ where 
$\cC^\circ$ is a subcategory of $\cC$ containing $\cC^\eq$.
Then the localization defines a functor $\cL:\Cat^+\to\Cat$
left adjoint to the functor $\cC\mapsto(\cC,\cC^\eq)$ 
\footnote{This is the most general notion of localization in infinity categories. We called it Dwyer-Kan localization, see~\cite{H.L}, in order to stress the origin of this notion.}.

\subsubsection{}
Let $f:\cC\to\cD$ be a functor. Let $q_\cC:\cC\to\cC'$ and
$q_\cD:\cD\to\cD'$ be localizations~\footnote{Without loss of generality we can think that $\cD=\cD'$.} defined by subcategories $W_\cC$ and $W_\cD$ of $\cC$ and $\cD$ respectively. Denote $\cE'_f$ the localization of $\cE_f$
with respect to $W_\cC\cup W_\cD$.

\begin{dfn}
\label{dfn:left}
If $\cE'_f$ is a cocartesian fibration over $[1]$,
we define a left derived functor $\Left f:\cC'\to\cD'$ as the functor classifying the cocartesian fibration $\cE'_f$.
\end{dfn}

Right derived functor is defined similarly. Denote 
$\cF'_f$ the localization of the cartesian fibration 
$\cF_f$ with respect to $W_\cC\cup W_\cD$.

\begin{dfn}
\label{dfn:right}
If $\cF'_f$ is a cartesian fibration over $[1]$,
we define a right derived functor $\Right f:\cC'\to\cD'$ as the functor classifying the cartesian fibration 
$\cF'_f\to[1]$.
\end{dfn}

We will discuss in Section~\ref{sec:main} the 
existence of derived functors defined in \ref{dfn:left} and \ref{dfn:right} and compatibility of these notions
with the known  ones.

The application to deriving adjunction is immediate.

\begin{prp}
\label{prp:adj}
Let $f:\cC\rlarrows\cD:g$ be an adjoint pair of functors.
Let $q:\cC\to\cC'$ and $q:\cD\to\cD'$ be localizations. Assume that the derived functors $\Left f$ and $\Right g$
exist. Then they form an adjoint pair.
\end{prp}

In fact, $\cE_f=\cF_g$ and so $\cE'_f=\cF'_g$.
\qed

\section{Correspondences and Kan extension}
\label{sec:corr-ext}

The key to understanding derived functors lies in 
the category of correspondences $\Cor$ and its full subcategory $\Cor^\coc$ of cocartesian correspondences.

\subsection{Correspondences}

Recall that a correspondence from $\cC$ to $\cD$ is given 
by a functor $\cC\times\cD^\op\to\cS$, where $\cS$ is the category of spaces. Equivalently, a correspondence can be defined as a functor $\cD\to P(\cC)$ to presheaves of $\cC$,
or, vice versa, as a functor $\cC^\op\to P(\cD^\op)$.

Equivalently, a correspondence from $\cC$ to $\cD$ can be 
encoded into a functor $p:\cE\to[1]$, together with the equivalences $\cC\stackrel{\sim}{\to}\cE_0$, 
$\cD\stackrel{\sim}{\to}\cE_1$, see \cite{L.T}, 2.3.1.3
and \cite{H.L}, 9.10. A detailed proof of this fact
(in a greater generality) can be found in~\cite{H.EY},
 Sect.~8.

We define $\Cor=\Cat_{/[1]}$ the category of correspondences and $\Cor^\coc$ its full subcategory
spanned by cocartesian fibrations over $[1]$. Note that 
arrows in $\Cor^\coc$ do not necessarily preserve cocartesian arrows. 

\subsubsection{}

The map $\{0\}\sqcup\{1\}\to[1]$ defines, by the base change, a functor $\partial=(\partial_0,\partial_1):\Cor\to\Cat\times\Cat$. Correspondences can be composed: if
$X,Y\in\Cor$ and $\cD:=\partial_1(X)=\partial_0(Y)$, the composition $Y\circ X$ is defined as follows. 
The category $Z:=X\sqcup^\cD Y$ is endowed with a map to
$[2]$; one defines $Y\circ X$ as the base change of $Z$
with respect to $\delta^1:[1]\to[2]$.

We denote by the same letter the restriction of 
$\partial$ to $\Cor^\coc$; the fiber at $(\cC,\cD)$ is denoted $\Cor^\coc_{\cC,\cD}$. The assignment $(\cC,\cD)\mapsto
\Cor^\coc_{\cC,\cD}$ is covariant in $\cD$ and contravariant in $\cC$. In fact, a map $f:\cC\to\cC'$
defines a functor $f^*:\Cor^\coc_{\cC',\cD}\to
\Cor^\coc_{\cC,\cD}$ carrying $X\in\Cor^\coc_{\cC',\cD}$
to the composition $X\circ\cE_f$. Similarly, $g:\cD\to\cD'$
defines a functor $g_*:\Cor^\coc_{\cC,\cD}\to
\Cor^\coc_{\cC,\cD'}$ carrying $X\in\Cor^\coc_{\cC,cD}$
to the composition $\cE_g\circ X$. Actually, the map
$\partial:\Cor^\coc\to\Cat\times\Cat$ is a bifibration in the sense of \cite{H.EY}, 2.2.6; the proof is given in 
\ref{sss:cor-coc-bif} below, after a reminder of the relevant notions. 

\subsubsection{Bifibrations. Lax bifibrations}
\label{sss:laxbi}

Let $p:X\to B\times C$ be a functor. We denote $p_B$ and 
$p_C$ its compositions with the projections to $B$ and $C$,
and $X_{b,c}$ (resp., $X_{b\bullet}$ or $X_{\bullet c}$)
the fiber of $p$ at $(b,c)$ (resp, the fiber of $p_b$ at $b$ or the fiber of $p_C$ at $c$).

\begin{Dfn}
The functor $p:X\to B\times C$ is called {\sl a lax bifibration} if
\begin{itemize}
\item[(1)] $p_B$ is a cartesian fibration and $p_C$ is a cocartesian fibration.
\item[(2)] $p$ is a morphism of cartesian fibrations over $B$,
as well as of cocartesian fibrations over $C$.
\end{itemize}
\end{Dfn}

If $p:X\to B\times C$ is a lax bifibration, the restriction
$p_C:X_{b\bullet}\to C$ is a cocartesian fibration and
$p_B:X_{\bullet c}\to B$ is a cartesian fibration.

\begin{lem}
\label{lem:criteria-bif}
The following conditions on $p$ are equivalent.
\begin{itemize}
\item[(1)] $p_C$ is a cocartesian fibration, $p$ is a map of cocartesian fibrations over $C$, and for any $c\in C$ 
the restriction of $p_B$ to $X_{\bullet c}$ is a cartesian fibration.
\item[(2)] $p_B$ is a cartesian fibration, $p$ is a map of cartesian fibrations over $B$, and for any $b\in B$ 
the restriction of $p_C$ to $X_{b\bullet}$ is a cocartesian fibration.
\item[(3)] $p$ is a lax bifibration.
\end{itemize}
\end{lem}
\begin{proof} Let us show that (1) implies (2). 

In the formulas below we use the following notation. Given $p:X\to B\times C$,
$x,y\in X$ and $f:p(x)\to p(y)$, we denote $\Map^f(x,y)$ the fiber of the map
$\Map_X(x,y)\to\Map_{B\times C}(p(x),p(y))$ at $f$; we denote $x\to f_!(x)$ the cocartesian lifting of $f$, and $f^!(y)\to y$ the cartesian lifting of $f$ (if it exists).

For arrows $\beta:b\to b'$ in $B$ and $\gamma:c\to c'$ in $C$, and for $x\in X_{b,c}$ and $y\in X_{b',c'}$ one has
$\Map^{(\beta,\gamma)}(x,y)=\Map^{(\beta,c')}(\gamma_!x,y)=
\Map^{(b,c')}(\gamma_!x,\beta^!y)$. This formula immediately implies that the map 
$\beta^!(y)\to y$ is locally $p_B$-cartesian. Since composition of such arrows has the same form, $p_B$ is a cartesian fibration. The rest
of the properties, as well as the opposite implication,
are clear.
\end{proof}

A  lax bifibration can be described as a functor 
$B^\op\to\Cat_{/C}$ carrying each object of $B^\op$ to a 
cocartesian fibration over $C$, or, equivalently, as a 
functor $C\to\Cat_{/B}$ carrying each object of $C$ to a 
cartesian fibration over $B$.

\begin{lem}
\label{lem:dec-laxbifib}
Let $p:X\to B\times C$ be a lax bifibration. Then 
\begin{itemize}
\item $p^{[1]}:X^{[1]}\to B^{[1]}\times C^{[1]}$ is also a lax bifibration.
\item The fiber $X^{[1]}_{\beta,\gamma}$ of $p^{[1]}$, with
$\beta:b\to b'$ and $\gamma:c\to c'$ arrows in $B$ and $C$,
is naturally equivalent to the fiber product
\begin{equation}\label{eq:bifib-pres}
X_{b,c}\times_{X_{b,c'}}X^{[1]}_{b,c'}\times_{X_{b,c'}}X_{b',c'}.
\end{equation}
\end{itemize}
 \end{lem}
\begin{proof}The functor $\Fun([1],\_)$ is well-known to preserve cartesian and cocartesian fibrations, see, for instance, \cite{H.lec}, 9.6.3.  This implies the first claim of the lemma. The fiber
$X^{[1]}_{\beta,\gamma}$ is the category of sections
$\Fun_{B\times C}([1],X)$, with $[1]\to B\times C$ given by $(\beta,\gamma)$. Recall that, for a cocartesian 
fibration $p:E\to[1]$ classified by a functor
$E_0\to E_1$, $E_i=p^{-1}(i)$ for $i=0,1$, the category of sections is described by the formula
\begin{equation}
\label{eq:sections-cocartesian}
\Fun_{[1]}([1],E)=E_0\times_{E_1}E_1^{[1]},
\end{equation}
see, for instance,~\cite{H.lec}, 9.8.5. We get the formula 
(\ref{eq:bifib-pres}) for lax  bifibrations by applying  
the formula~(\ref{eq:sections-cocartesian}) (and a similar formula for cartesian fibrations) twice. 
First of all, since $p_B$ is a cartesian fibration, we get
\begin{equation}
\label{eq:bifib-pres2}
X^{[1]}_{\beta\bullet}=X^{[1]}_{b\bullet}
\times_{X_{b\bullet}}X_{b'\bullet}.
\end{equation}
The category $X^{[1]}_{\beta,\gamma}$ that we wish to describe is the fiber at $\gamma$ of 
(\ref{eq:bifib-pres2}), so that
$$X^{[1]}_{\beta,\gamma}=X^{[1]}_{b,\gamma}\times_{X_{b,c'}}X_{b',c'}.
$$
Finally,  since the restriction of $p_C$ to 
$X_{b\bullet}$ is a cocartesian fibration, one gets
$$ X^{[1]}_{b,\gamma}=X_{b,c}\times_{X_{b,c'}}X^{[1]}_{b,c'}.$$
This proves the lemma.
\end{proof}

\begin{dfn}A map $p:X\to B\times C$ is called a bifibration
if it is a lax bifibration and the following equivalent
extra conditions are fulfilled.
\begin{itemize}
\item For any $\beta:b\to b'$ the functor 
$\beta^!:X_{b'\bullet}\to X_{b\bullet}$ preserves 
$p_C$-cocartesian arrows.
\item For any $\gamma:c\to c'$ the functor 
$\gamma_!:X_{\bullet c}\to X_{\bullet c'}$ preserves 
$p_B$-cartesian arrows.
\end{itemize}
\end{dfn}
 Thus, a bifibration $p:X\to B\times C$ is classified by a functor $B^\op\times C\to\Cat$.

\subsubsection{}
\label{sss:cor-coc-bif}
Let us calculate $\Map(\cE_f,\cE_{f'})$ for 
$f:\cC\to\cD$ and $f':\cC'\to\cD'$. We have
$\cE_f=(\cC\times[1])\sqcup^\cC\cD$, so the result can be calculated as a fiber product. We have
$$\Map(\cC\times[1],\cE_{f'})=\Map(\cC,\Fun_{[1]}([1],\cE_{f'}))=
\Map(\cC,\cC'\times_{\cD'}{\cD'}^{[1]}),$$
see the formula~(\ref{eq:sections-cocartesian}).
Therefore,
\begin{equation}
\label{eq:fun-cor-coc}
\Map(\cE_f,\cE_{f'})= 
\Map(\cC,\cC')\times_{\Map(\cC,\cD')}\Map(\cC,{\cD'}^{[1]})
\times_{\Map(\cC,\cD')}\Map(\cD,\cD').
\end{equation}

The formula~(\ref{eq:fun-cor-coc}) shows that, for any
$f:\cC\to\cC'$, $g:\cD\to\cD'$ and $X\in\Cor^\coc_{\cC',\cD}$, the morphism
$X\circ\cE_f\to X$ is a cartesian lifting of $f$,
whereas $X\to\cE_g\circ X$ is a cocartesian lifting of $g$.
This proves the following.
\begin{Lem}
The map $\partial:\Cor^\coc\to\Cat\times\Cat$ is a bifibration.
\end{Lem}\qed

\subsection{Kan extensions}

Let $f:\cC\to\cD$ be a functor. A right extension of
$f$ with respect to $q:\cC\to\cC'$ is a functor
$f':\cC'\to\cD$, endowed with a morphism
of functors $\theta:f'\circ q\to f$. A right Kan extension 
is, by definition, a terminal object in the category of
right extensions, see \cite{L.T}, 4.3.3 for the version for 
quasicategories.

The category $\RE_q(f)$ is defined, therefore, as the 
fiber product
\begin{equation}
\label{eq:reqf}
\RE_q(f)=\Fun(\cC',\cD)\times_{\Fun(\cC,\cD)}
\Fun(\cC,\cD^{[1]})\times_{\Fun(\cC,\cD)}\{f\},
\end{equation}
where the map $\Fun(\cC',\cD)\to\Fun(\cC,\cD)$ is given by 
composition with $q$.

\begin{dfn}
\label{dfn:rightkan}
A right Kan extension of $f:\cC\to\cD$ with respect to 
$q:\cC\to\cC'$ is a terminal object in $\RE_q(f)$.
\end{dfn}

\subsubsection{}
\label{sss:reop=cor}

Fix $f:\cC\to\cD$. Let $\partial_f:\Cor^\coc_{f/}\to\Cat_{\cC/}\times\Cat_{\cD/}$ be induced by 
$\partial:\Cor^\coc\to\Cat\times\Cat$.
Denote 
 $(\Cor^\coc_{/f})_{q}$ the fiber of $\partial_f$
at $(q,\id_\cD)$.

In \ref{ss:reop=cor}  below we prove the following.
 
\begin{Prp}\label{prp:re-cor}
One has a natural equivalence $\RE_q(f)^\op=(\Cor^\coc_{f/})_{q}$.
\end{Prp} \qed

This allows one to redefine a right Kan extension
of $f:\cC\to \cD$ with respect to $q:\cC\to\cC'$ as
the initial object in $(\Cor^\coc_{f/})_{q}$.

In other words, a right
extension of $f$ along $q$ is an arrow 
$\theta:\cE_f\to\cE_{f'}$ in $\Cor^\coc$, with 
$\partial_0(\theta)=q$, $\partial_1(\theta)=\id_\cD$.

A right Kan extension is an initial object in the category
of such $\theta$'s.

\subsection{Construction of equivalence}
\label{ss:reop=cor}

We will now identify the bifibration $\Cor^\coc$
given by a functor $\Cat^\op\times\Cat\to\Cat$
carrying $(\cC,\cD)$ to $\Cor^\coc_{\cC,\cD}$, with 
the opposite to the internal Hom in $\Cat$.

\begin{lem}
\label{lem:fun=corcoc}
There is an equivalence
$$\Fun(\cC,\cD)^\op=\Cor^\coc_{\cC,\cD},$$
functorial in $\cC$, $\cD$.
\end{lem} 

To deduce Proposition~\ref{sss:reop=cor}, we apply
Lemma~\ref{lem:dec-laxbifib}  to $\Cor^\coc$
and the arrow $[1]\to\Cat\times\Cat$ given by the pair 
$(q,\id_\cD)$. 
\begin{proof}[Proof of \ref{lem:fun=corcoc}]
Recall an important way of presentation of categories
going back to Rezk's CSS model structure for 
$\infty$-categories.

The embedding $\Delta\to\Cat$ defines a functor $\cN:\Cat\to P(\Delta)=\Fun(\Delta^\op,\cS)$ assigning to a category $\cC$ its ``Rezk nerve'', a simplicial space carrying
$[n]$ to $\Map([n],\cC)$. The functor
$\cN$ is fully faithful; its image consists of simplicial spaces that are Segal and complete. Moreover, this embedding has a left adjoint 
$L:\Fun(\Delta^\op,\cS)\to\Cat$ which is a localization
~\footnote{A localization having fully faithful right adjoint is called Bousfield localization.}.

Let us describe the simplicial space $\cN(\Fun(\cC,\cD))$.
One has
$$ \Map([n],\Fun(\cC,\cD))=\Map_{[n]}(\cC\times[n],\cD\times[n]).$$
Lemma~\ref{lem:criteria-bif}, (1), shows that this is the space of lax bifibrations $p:X\to[n]\times[1]$ satisfying the conditions
\begin{equation}
\label{eq:boundary}
X_{\bullet0}=[n]\times\cC\textrm{ and }X_{\bullet1}=[n]\times\cD.
\end{equation}

Let us now describe $\cN(\Cor^\coc_{\cC,\cD})$. We use the cartesian 
version of Grothendieck construction and 
Lemma~\ref{lem:criteria-bif}, (2), to describe 
$\Map([n],\Cor^\coc_{\cC,\cD})$ as the space of lax bifibrations
$p:X\to[n]^\op\times[1]$
satisfying the conditions (\ref{eq:boundary}).
Using the canonical isomorphism $[n]=[n]^\op$, we get 
the required equivalence.
\end{proof}
 
\section{Main results}
\label{sec:main}

The results listed below are mostly direct consequences of the constructions of Section~\ref{sec:corr-ext}.

\begin{LEM}
Let $f:\cC\to\cD$ be as above, $q_\cC:\cC\to\cC'$ and 
$q_\cD:\cD\to\cD'$ be localizations. Let $\cE_f'$ be the 
localization of $\cE_f$ described in \ref{ss:def-derived}
and $\theta:\cE_f\to\cE_{f'}$ be a right extension of $q_\cD\circ f$ along $q_\cC$. Then $\theta$ uniquely factors
through $\cE'_f$.
\end{LEM}
\qed

We denote the unique map $\cE_f'\to\cE_{f'}$ as $\theta'$.

\begin{CRL}
Let $f:\cC\to\cD$ admit a left derived functor $\Left f$.
Then the canonical equivalence $\cE_f'\to\cE_{\Left f}$
presents $\cE_{\Left f}$
as a right Kan extension of $q_\cD\circ f$ along $q_\cC$.
\end{CRL}
That is, our left derived functors are right Kan extensions,
and right derived functors are left Kan extension. 
\footnote{It is very easy to see that they are actually absolute Kan extensions.} 
Note that our derived functors may not exist, even when
Kan extensions exist. Note also, that

\begin{PRP}
\label{prp:preserves}
Let $f:\cC\to\cD$ be a functor, $W_\cC$ and $W_\cD$
subcategories of $\cC$ and $\cD$ respectively, and let
$f(W_\cC)\subset W_\cD$. Then the composition 
$q\circ f:\cC\to\cD'$ factors uniquely through a functor 
$f':\cC'\to\cD'$ which is both left and right derived 
functor of $f$.
\end{PRP}
\begin{proof}
We have to verify that the canonical map
$\theta':\cE_f'\to\cE_{f'}$ is an equivalence.
By universality of localization, 
\begin{equation}
\cE_f'=\cC'\sqcup^\cC(\cC\times[1])\sqcup^\cC\cD',
\end{equation}
with the map $\cC\to\cD'$ given by a composition $q\circ f$.
We will show that the natural map from the expression above
to $\cE_{f'}=(\cC'\times[1])\sqcup^{\cC'}\cD'$ is an equivalence. It is enough, for any $\cX$, to prove that
the induced map
\begin{equation}
\Fun(\cE_{f'},\cX)\to\Fun(\cE'_f,\cX)
\end{equation}
is an equivalence. It is easy to see that, by universality
of localization, both identify with the full subcategory
of $\Fun(\cE_f,\cX)$ spanned by the functors carrying
$W_\cC$ and $W_\cD$ to equivalences in $\cX$.
\end{proof}
 
\subsection{Existence}
\label{sss:existence}
We will now prove that derived functors exist reasonably often.

Keeping the previous notation, let now $i:\cC_0\to\cC$ be a functor. We denote $W_0=i^{-1}(W_\cC)$ and we assume that
the composition $f\circ i:\cC_0\to\cD$ preserves weak equivalences. The commutative diagram
\begin{equation}
\xymatrix{
&\cC_0\ar^{f\circ i}[r]\ar^i[d] &\cD\ar^\id[d]\\
&\cC\ar^f[r] &\cD
}
\end{equation}
gives rise to a map $\alpha:\cE_{f\circ i}\to\cE_f$ in $\Cor$, and, after the localization, to a map $\alpha':\cE'_{f\circ i}\to\cE'_f$.

\begin{prp}
\label{prp:exist}
Let $f:\cC\to\cD$, $W_\cC$ and $W_\cD$ be as above.
Assume there exists a functor $i:\cC_0\to\cC$ satisfying the
following properties. 
\begin{itemize}
\item The composition $f\circ i:\cC_0\to\cD$ preserves weak equivalences.
\item The map $i$ induces an equivalence $i':\cC_0'\to\cC'$. 
\item A right Kan extension of $q_\cD\circ f$ along
$q_\cC$ exists. It is given by a map $\theta:\cE_f\to\cE_{f'}$, where $f':\cC'\to\cD'$, such that the composition
$\theta'\circ\alpha':\cE'_{f\circ i}\to\cE'_f\to\cE_{f'}$
is an equivalence.
\end{itemize}
Then the left derived functor $\Left f$ exists (and can be calculated as $(f\circ i)'\circ {i'}^{-1}$).
\end{prp}
\begin{proof}
We have to verify that the canonical map 
$\theta':\cE_f'\to\cE_{f'}$ obtained by the localization of 
the right Kan extension $\theta:\cE_f\to\cE_{f'}$, is an equivalence. Since $\theta'\circ\alpha'$ is an equivalence,
its inverse, composed with $\alpha':\cE'_{f\circ i}\to\cE'_f$, yields a map $\cE_{f'}\to\cE'_f$ 
in the opposite direction.

Both compositions are equivalent to identity as 
 $\cE_{f'}$ is universal in $(\Cor^\coc_{f/})_q$ and 
$\cE'_f$ is universal in $(\Cor_{f/})_q:=
\Cor_{f/}\times_{\Cat_{\cC/}\times\Cat_{\cD/}}\{(q,\id_\cD)\}$. 
\end{proof}

Let now $f:\cC\to\cD$ be a functor between model categories. Applying \ref{prp:exist} to $\cC_0$ the subcategory of fibrant (resp., cofibrant) objects in $\cC$,
we get the following.

\begin{crl}
\label{crl:exist}
Left (resp., right) Quillen functors admit a left (resp., right) derived functor in the sense of Definition~\ref{dfn:left} (resp., \ref{dfn:right}).
\end{crl}\qed

\subsection{Functoriality}
\label{ss:functoriality}

\subsubsection{Flat fibrations}
Recall that a functor $p:\cE\to[2]$ is called a flat fibration, if the natural map
$$ \cE_{\{0,1\}}\sqcup^{\cE_1}\cE_{\{1,2\}}\to\cE$$
where $\cE_i$, resp., $\cE_{\{i,j\}}$, are defined  by
base change of $\cE$ with respect to $\{i\}\to[2]$,
resp., $\{i,j\}\to[2]$, is an equivalence. More generally, a map $p:\cE\to B$
is a flat fibration if any base change of $p$ with respect
to $[2]\to B$ is flat in the above sense. An important
property of flat fibrations, see~\cite{L.HA},B.4.5, says
that, if $\cE\to\ B$ is flat, the base change functor
$\Cat_{/B}\to\Cat_{/\cE}$ preserves colimits.

\subsubsection{}
Derived functor of a composition is not, in general, the
composition of derived functors. It is interesting to see 
what is going on here, in terms of correspondences. A composable pair of functors 
$f:\cE_0\to\cE_1,\ g:\cE_1\to\cE_2$ is given by 
a cocartesian fibration $p:\cE\to[2]$. Given a subcategory
$\cW\subset\cE$ over $[2]^\eq=\{0\}\cup\{1\}\cup\{2\}$,
the localization $\cE'$ is flat over $[2]$, as the universal property of localization gives the presentation
$$
\cE'=\cE'_0\sqcup^{\cE_0}\cE_{\{0,1\}}\sqcup^{\cE_1} \cE'_1\sqcup^{\cE_1}\cE_{\{1,2\}}\sqcup
^{\cE_2}\cE'_2,
$$
as well as presentations
$$
\cE'_{0,1}=\cE'_0\sqcup^{\cE_0}\cE_{\{0,1\}}\sqcup^{\cE_1}
\cE'_1
\textrm{ and }
\cE'_{1,2}=\cE'_1\sqcup^{\cE_1}\cE_{\{1,2\}}\sqcup^{\cE_2}
\cE'_2.
$$

If $\Left f$ and $\Left g$ exist,
$\cE'$ is therefore a cocartesian fibration, that is, it
is classified by the pair of functors $\Left f$, $\Left g$.

The base change of the localization map $\cE\to\cE'$ with respect to $\delta^1:[1]\to[2]$ yields a map 
\begin{equation}
\cE_{g\circ f}\to\cE_{\Left g\circ\Left f},
\end{equation} 
which induces a map 
\begin{equation}
\label{eq:leftgf}
\cE'_{g\circ f}\to\cE_{\Left g\circ\Left f}.
\end{equation}
This map is not necessarily an equivalence
~\footnote{Thus, the localization does not always commute 
with the base change.}. 

If the canonical map (\ref{eq:leftgf}) is an equivalence,
the left derived functor of $g\circ f$ is defined and
$\Left g\circ\Left f=\Left(g\circ f)$.

\subsubsection{Deriving a  family of functors}
\label{sss:family}

Given a functor $F:B\to\Cat$, we can convert it into
a cocartesian fibration $p:\cE\to B$. Given a subcategory
$W$ in $p^{-1}(B^\eq)$, we can define $p':\cE'\to B$
as the localization of $\cE$ with respect to $W$.

\begin{Dfn}
We say that the functor $F$ is left derivable with respect to $W$ if 
\begin{itemize}
\item $p':\cE'\to B$ is a cocartesian fibration.
\item For each $a:[1]\to B$ the base change of $\cE\to\cE'$
with respect to $a$ remains a localization.
\end{itemize}
\end{Dfn}

Thus, a derivable functor $F:B\to\Cat$ gives rise to a new functor $F':B\to\Cat$ defined by the formulas
$$ F'(b)=F(b)',\ F'(\alpha)=\Left F(\alpha).$$

It is obvious that, when $F:B\to\Cat$ is left derivable, 
one has an equivalence
$$ \Left F(\gamma)\stackrel{\sim}{\to}\Left F(\beta)
\circ\Left F(\alpha)$$
for any commutative triangle with edges $\alpha,\beta$
and $\gamma=\beta\circ\alpha$ in $B$. The following result shows that the converse is also true.

\begin{prp} 
\label{prp:functor}Let $p:\cE\to B$ be a cocartesian fibration
classified by a functor $F:B\to\Cat$. Let
$W$ be a subcategory of $p^{-1}(B^\eq)$. Assume that
\begin{itemize}
\item[1.] For any $a:[1]\to B$ the composition 
$F\circ a:[1]\to\Cat$ defines a functor having a left derived functor with respect to $W$.
\item[2.] For any $b:[2]\to B$ with edges
$\alpha,\beta$
and $\gamma=\beta\circ\alpha$, the natural map
$$  \Left F(\gamma)\stackrel{\sim}{\to}\Left F(\beta)
\circ\Left F(\alpha)$$
is an equivalence.
\end{itemize}
Then $F$ is left derivable with respect to $W$.
\end{prp}
\begin{proof}
The proof will proceed as follows. First of all, we
will verify the claim for $B=[n]$. Then we will deduce
it for a general $B$, using a presentation of $B$ as a colimit.

The claim is vacuous for $B=[0], [1]$. Let us verify it
for $B=[n]$. The maps $s:[1]\to[n]$ and $t:[n-1]\to [n]$
given by the formulas $s(0)=0,\ s(1)=1,\ t(i)=i+1$, are 
flat. $\cE'$ is obtained by localizing $\cE$ with respect to
$W=\sqcup W_i$ where $W_i\subset\cE_i=p^{-1}(\{i\})$,
$i\in[n]$. This allows to present $\cE'$ as a colimit
of the diagram
\begin{equation}
\xymatrix{
&{} &{\sqcup_i\cE_i}\ar[dr]\ar[dl] &{}\\
&\cE &{} &{\sqcup_i\cE'_i}.
}
\end{equation}
The base change with respect to flat $s$ and $t$ preserves colimits. We deduce that, by induction, $p':\cE'\to[n]$
is a cocartesian fibration.

Now, given $c:[1]\to[n]$, we have to verify that the 
base change of $\cE\to\cE'$ with respect to $c$,
$\cE_c\to\cE'_c$, is a localization. By induction, it is sufficient to assume that $c(0)=0,\ c(1)=n$. 
Look at the commutative triangle $\gamma=\beta\circ\alpha$
in $[n]$ with $\gamma:0\to n$, $\alpha:0\to 1$ and 
$\beta:1\to n$. We denote by $a:[1]\to[n]$ and $b:[1]\to
[n]$ the maps corresponding to  $\alpha$ and $\beta$.

By induction we can assume that $\cE_a\to \cE'_a$ and 
$\cE_b\to \cE'_b$ are localizations, which gives,
by Condition 2 applied to the commutative triangle
$\gamma=\beta\circ\alpha$, that $\cE_c\to\cE'_c$ is a localization.
 
We now treat the case of arbitrary $B$. 
Recall that the Grothendieck construction provides a 
canonical identification between the category of 
cocartesian fibrations $\Coc(B)$ over $B$ and the category 
of functors $\Fun(B,\Cat)$. This implies that, for 
$B=\colim B_i$,
one has an equivalence $\eta:\Coc(B)\to\lim\Coc(B_i)$.
The map $\eta$ is defined by the compatible collection of 
base change maps $\Coc(B)\to\Coc(B_i)$ for each $B_i\to B$. 
Let us describe a quasi-inverse map 
$\theta:\lim\Coc(B_i)\to\Coc(B)$. Given a  compatible collection  of cocartesian fibrations 
$\cE_i\to B_i$, it assigns a cocartesian fibration $p:\cE\to B$ endowed with a compatible collection of equivalences
$\cE_i\to B_i\times_B\cE$. Since $p:\cE\to B$ is flat, 
the fiber product $\times_B\cE$ preserves colimits, so
the collection of maps $\cE_i\to\cE$ is a colimit diagram.
Thus, $\theta(\{\cE_i\})=\colim\cE_i$.

Thus, given a compatible
collection of cocartesian fibrations $p_i:\cE_i\to B_i$,
one has $\cE_i=\cE\times_BB_i$ where $B=\colim B_i$ and 
$\cE=\colim\cE_i$.

Let now $\Delta_{/B}$ denote the full subcategory
of $\Cat_{/B}$ spanned by the arrows $[n]\to B$. 
Denote  $\pi:\Delta_{/B}\to\Cat$ the functor carrying
$a:[n]\to B$ to $[n]\in\Cat$. It is standard that $B=\colim(\pi)$~\footnote{$\Cat$ is a full subcategory of $P(\Delta)$ and $B=\colim(\pi)$ where the colimit is calculated in $P(\Delta)$. To calculate the colimit in $\Cat$, one has to make a complete Segal replacement, which is not needed since $B$ is already complete and Segal simplicial space.}.

Therefore, the functor $\rho:\Delta_{/B}\to\Cat$ carrying
$a:[n]\to B$ to $\cE_a=[n]\times_B\cE$, has the colimit
$\cE=\colim(\rho)$. Recall that $\cE$ has a marking defined by the subcategory $W\subset p^{-1}(B^\eq)$. The arrows
in $\Delta_{/B}$ preserve the markings, so $\rho$ is actually a functor from $\Delta_{/B}$ to $\Cat^+$, the category of marked categories, see~\ref{sss:loc}. We denote
$\rho':\Delta_{/B}\to\Cat$ the composition of $\rho$ with the localization $\cL:\Cat^+\to\Cat$. Since localization commutes with colimits, we have $\colim(\rho')=\cE'$.
This implies that for any $a:[n]\to B$ the localization
of $\cE_a=[n]\times_B\cE$ is $\cE'_a=[n]\times_B\cE'$.
\end{proof}

\subsubsection{Deriving a family of adjoint pairs of functors}

A family of adjoint pairs of functors is just a functor
$F:B\to\Cat$, such that for each arrow $\alpha:b\to b'$ in 
$B$ the functor $F(\alpha)$ has a right adjoint. 
equivalently, this means that the corresponding cocartesian 
fibration $p:\cE\to B$ is also a cartesian fibration.

For $W\subset p^{-1}(B^\eq)$ let $p':\cE'\to B$ be obtained by localizing $p:\cE\to B$ with respect to $W$.

Proposition~\ref{prp:functor} implies the following.

\begin{Crl}
Let, in the notation of Proposition~\ref{prp:functor},
$p:\cE\to B$ be a cartesian and cocartesian fibration. Assume as well 
that, apart of conditions of ~\ref{prp:functor}, for each
$a:[1]\to B$ the base change $\cE\times_B[1]$ has a cartesian and cocartesian localization. Then $p':\cE'\to B$ is also a cartesian cocartesian localization(and therefore defines a family
of derived adjoint pairs of functors).
\end{Crl}\qed

\subsection{Deligne's definition}
\label{ss:deligne}
Let $f:\cC\to \cD$ be functor  
between conventional categories, and let $W$ be a  
muliplicative system  in $\cC$ satisfying the right calculus 
of fractions in the sense of Gabriel-Zisman~\cite{GZ}, 
Section 2.  

According to Deligne's approach~\cite{D}, Def. 1.2.1
\footnote{Deligne formulates this notion for triangulated categories.}, the left derived functor $\Left f$ is constructed as follows.

For $x\in\cC$ denote $L_x$
the category of arrows $s:x'\to x$ in $\cC$ belonging to 
$W$. Define a functor 
\begin{equation}
f':\cC^\op\times \cD\to\Set
\end{equation}
by the formula
$$
(x,y)\mapsto\colim\{\phi_{x,y}:L_x^\op\to\Set\},
$$ 
where $\phi_{x,y}(s:x'\to x)=\Hom_\cD(f(x'),y)$.
The functor $f'$ defines a correspondence from $\cC$
to $\cD$ which we denote $\cE^\Del$.

The left derived functor $\Left f$ is defined whenever 
$\cE^\Del$ is a cocartesian fibration.

In this case $\cE^\Del$ defines a functor $\cC\to \cD$ carrying
$W$ to isomorphisms, and so having a unique factorization through the localization $\cC'$ of $\cC$.

Let $\cE_f$ be the cocartesian fibration classified by the functor $f$. One has an obvious map $\cE_f\to \cE^\Del$
in $\Cor$. 
let $\cE_f'$
be the localization of $\cE$ with respect to $W$. One has a map $\delta:\cE^\Del\to\cE'_f$ defined by the compatible collection of maps 
$\Hom_\cD(f(x'),y)\to\Hom_{\cE_f'}(x,y)$ 
carrying $\alpha:f(x')\to y$ to the composition
$$x\stackrel{s^{-1}}\to x'\stackrel{\alpha}{\to}y$$
 in $\cE'_f$. The functor $\delta$ induces a functor
$\delta':(\cE^\Del)'\to\cE_f'$ from the localization of 
$\cE^\Del$ with respect to $W$, to $\cE_f'$.
This map is fully faithful, as $(\cE_f,W)$ satisfies the
right calculus of fractions. Therefore, Deligne's definition 
of derived functor, in case $W$ satisfies a calculus
of fractions, coincides with the one defined
in \ref{dfn:left}.

\subsection{$\infty$-localization versus conventional localization}
\label{ss:infty-conv}

Our definition of derived functor makes sense in  both 
contexts, as we only use universal property of 
localizations. Moreover, if a derived functor  exists in the infinity setting, its conventional truncation gives a
derived functor in the conventional setting.

\begin{prp} Lef $f:\cC\to\cD$ be a functor between
conventional categories. Let $\cC'$ and $\cD'$ be their 
$\infty$-localizations, and $\Ho(\cC),\Ho(\cD)$ their
conventional truncations obtained by applying the functor 
$\pi_0$ to all function spaces. Assume $f':\cC'\to\cD'$ is a 
left derived functor of $f$ in the sense of 
Definition~\ref{dfn:left}. Then the induced conventional functor $\Ho(f'):\Ho(\cC')\to\Ho(\cD')$ is a left derived functor in the conventional setting.
\end{prp}
The claim immediately follows from the following observation.

\begin{lem}Let $f:\cC\to\cD$ be a functor, $p:\cE_f\to[1]$ be the corresponding cocartesian fibration.
Then the induced map $\Ho(p):\Ho(\cE)\to[1]$ is a cocartesian fibration classified by the functor
$\Ho(f):\Ho(\cC)\to\Ho(\cD)$.
\end{lem}
\begin{proof}
The commutative square 
$$
\xymatrix{
&\cC\ar^f[r]\ar[d]&\cD\ar[d]\\
&\Ho(\cC)\ar^{\Ho(f)}[r]&\Ho(\cD)
}
$$
induces a map  $\cE_f\to\cE_{\Ho(f)}$ in $\Coc([1])$.
The map $\Ho(\cE_f)\to[1]$ is a cocartesian fibration
as the image in $\Ho(\cE_f)$ of any cocartesian arrow in 
$\cE_f$ is cocartesian. Therefore, a map of cocartesian
fibrations $\Ho(\cE_f)\to\cE_{\Ho(f)}$ is induced. It is
an equivalence as it induces an equivalence of the fibers.
\end{proof}

\end{document}